\documentclass{amsproc}

\usepackage{amsthm,array,amssymb,amscd,amsfonts,latexsym}
\usepackage[all]{xy}
\xyoption{poly}

\DeclareFontFamily{OT1}{wncyr}{\hyphenchar\font45 }
\DeclareFontShape{OT1}{wncyr}{m}{n}{%
   <5> <6> <7> <8> <9> gen * wncyr
   <10> <10.95> <12> <14.4> <17.28> <20.74>  <24.88>wncyr10}{}
\DeclareFontShape{OT1}{wncyr}{m}{it}{%
   <5> <6> <7> <8> <9> gen * wncyi
   <10> <10.95> <12> <14.4> <17.28> <20.74> <24.88> wncyi10}{}
\DeclareFontShape{OT1}{wncyr}{m}{sc}{%
   <5> <6> <7> <8> <9> <10> <10.95> <12> <14.4>
   <17.28> <20.74> <24.88>wncysc10}{}
\DeclareFontShape{OT1}{wncyr}{b}{n}{%
   <5> <6> <7> <8> <9> gen * wncyb
   <10> <10.95> <12> <14.4> <17.28> <20.74> <24.88>wncyb10}{}
\input cyracc.def
\def\rus{\usefont{OT1}{wncyr}{m}{n}\cyracc\fontsize{8}{10pt}\selectfont}

\DeclareMathSizes{9}{9}{7}{5} % This only is different.

\newtheorem{theorem}{Theorem}[section]
\newtheorem{lemma}[theorem]{Lemma}

\newtheorem{corollary}[theorem]{Corollary}
\newtheorem{question}[theorem]{Question}

\theoremstyle{definition}
\newtheorem{definition}[theorem]{Definition}
\newtheorem{example}[theorem]{Example}
\newtheorem{examples}[theorem]{Examples}

\theoremstyle{remark}
\newtheorem{remark}[theorem]{Remark}

\numberwithin{equation}{section}

\begin{document}

\title[Makar-Limanov, Derksen
invariants and finite automorphism groups]{On the Makar-Limanov,
 Derksen invariants, and\\ finite automorphism groups  of algebraic varieties}

\author{Vladimir  L. Popov}
\address{Steklov Mathematical Institute,
Russian Academy of Sciences, Gubkina 8, Moscow
119991, Russia}
\curraddr{}
\email{popovvl@mi.ras.ru}
\thanks{Supported by
 grants {\rus RFFI
08--01--00095}, {\rus N{SH}--1987.2008.1}, and the
program {\it Contemporary Problems of Theoretical
Mathematics} of the Russian Academy of Sciences, Branch
of Mathematics.}

\thanks{}


\date{January 8, 2010}

\dedicatory{To Peter Russell on the
occasion of his $70$th
birthday}

\begin{abstract}
 A simple method of constructing a big
stock of algebraic varieties with trivial Makar-Limanov
invariant is described, the Derksen invariant of some
varieties is computed, the generalizations of the Makar-Limanov
and Derksen invariants are introduced and discussed, and some
results on the Jordan property of automorphism groups
of algebraic varieties are obtained.
\end{abstract}

\maketitle

\section*{Introduction}
The subject
matter
of this note are automorphism groups of algebraic varieties.

 In Section 1 I discuss the Makar-Limanov and Derksen invariants.
 As is known, they have been first introduced as
the means for distinguishing the Koras-Russell threefolds
from affine spaces. Since then studying
varieties with certain properties of these invariants
(for instance, with trivial Makar-Limanov invariant)
became an independent line of research, see, e.g.,
\cite{Dai}, \cite{Dub}, \cite{FZ}, and
references therein. At the conference\,{\it Affine
Algebraic Geometry},
June 1--5, 2009,
Montreal,
I was surprised to find
that
 a simple general method of constructing a big stock
 of such varieties remained unnoticed by the experts. In
 Section 1
I expand my comment on this point made  after one of
the talks and give the related proofs and some illustrating
examples. Then
I consider the Derksen invariant and show that in many
cases in presence of an algebraic group action it coincides with
the coordinate algebra.~At the end of this section
I introduce and discuss the natural ge\-ne\-ralizations
of the Makar-Limanov and Derksen invariants.
In Section~2 some results on the Jordan property
of automorphism groups of algebraic varieties are obtained.

\vskip 2mm

\subsection*{\it Conventions and notation}\

\vskip 1mm

Below variety means algebraic variety. All varieties are taken over an
algebraically closed field $k$ of characteristic zero.
I use the standard conventions
of \cite{Bo} and \cite{Sp}
and the following notation.

\vskip 1.5mm
\begin{quote}
\begin{list}
{\renewcommand{\makelabel}{\entrylabel}}
{\topsep=1mm
\leftmargin=-6mm\itemindent=0.0mm
\labelwidth=1mm \labelsep=1.3mm }

\item{ } $A^*$ is the group of units of the commutative ring $A$ with identity.

\item{ } ${\bf M}_{n\times m}$ is the affine space of all
$n\times m$-matrices with entries in $k$.

\item{} ${\bf A}\!^1_*$ is the punctured affine line ${\bf A}\!^1\setminus \{0\}$.

    \item{ } ${\bf Z}_{>0}$ is the set of positive integers.

    \item{ }$|M|$ is the number of elements of
the set $M$.

    \item{ } ${\rm Rad}\,G$ is the radical of the linear algebraic group $G$.

\item{ } ${\rm Rad}_uG$ is the unipotent radical of the linear algebraic group $G$.

    \item{ } $(G, G)$ is the commutator subgroup of the group $G$.

 \item{ } $k[X]$ is the $k$-algebra of regular function on the variety~$X$.

\item{ } $k(X)$ is the field of rational function on the irreducible
    variety~$X$.

\item{ }${\rm T}_{x, X}$ is the tangent
space to
the variety
$X$ at the
point $x\in X$.

    \item{ } ${\rm Aut}(X)$ is the automorphism group of
    the variety
    $X$.

\item{ } ${\rm Bir}(X)$ is the group of
birational automorphisms of the irreducible
    variety~$X$.

\end{list}
\end{quote}

Given the varieties $X$ and $Y$ (not necessarily affine),
$k[X]$ and $k[Y]$ are naturally
identified  with the $k$-subalgebras of $k[X\times Y]$.
Recall that then
$k[X\times Y]$ is generated by $k[X]$ and $k[Y]$ and,
moreover,
$k[X\times Y]=k[X]\otimes_k k[Y]$, see \cite{SW}.
If $A$ and $B$ are the $k$-subalgebras of resp. $k[X]$
and $k[Y]$, then the subalgebra of $k[X\times Y]$
generated by $A$ and $B$ is
$A\otimes_k B$.

Below action of an algebraic group on an algebraic
variety means algebraic action. Homomorphism of
algebraic groups means  algebraic homomorphism.

Let $X$ be a variety endowed with
an action of an algebraic group $G$.
Then
the natural homomorphism
$\varphi\colon G\rightarrow  {\rm Aut}(X)$ defined by this action
is called {\it algebraic}
and $\varphi(G)$ is called
the {\it algebraic subgroup} of ${\rm Aut}(X)$. If
$\varphi$ is injective,
$\varphi(G)$ is identified with $G$ by means of $\varphi$.

\vskip 2mm

{\it Acknowledgement.} I am grateful to {\sc I. Dolgachev}, {\sc
Yu.\;Prokhorov}, and {\sc Yu. Zarhin}
for discussions on automorphism groups of surfaces.
 %%and, in particular,
%%drawing
%%my attention to
%%references %%\cite{C},
%%\cite{Do}, \cite{Mar}, .

\section{The Makar--Limanov and Derksen invariants}

\subsection{The Makar--Limanov invariant}\

 Recall that the {\it Makar-Limanov invariant}
 of a variety $X$ is the following $k$-subalgeb\-ra of $k[X]$:
 \begin{equation}\label{MLL}
 {\rm ML}(X):=\bigcap_{H}%%\varphi}
 k[X]^{H}%%{\rm Im}(\varphi)}
 \end{equation}
 where $H$
 in \eqref{MLL}
 runs over the images of all homomorphisms ${\bf G}_a\to{\rm Aut}(X)$.

Below is described a simple method of constructing
varieties whose  Makar-Limanov invariant is
trivial (i.e., equal to $k$). The starting point is

\begin{lemma}\label{GGG} %%Let $\,G$ be a
For every connected linear algebraic group $G$,
the  following are equivalent{\rm:}
\begin{enumerate}
\item[\rm(i)] $G$ has no nontrivial characters{\rm;}
\item[\rm(ii)] $G$ is generated by one-dimensional unipotent subgroups{\rm;}
    \item[\rm(iii)] $G$ is generated by  unipotent elements{\rm;}
    \item[\rm(iv)] ${\rm Rad}\,G= {\rm Rad}_uG$.
\end{enumerate}
\end{lemma}
\begin{proof} Let $G_0$ be the subgroup of $G$ generated
by all one-dimensional unipotent subgroups of $G$;
it is normal
and,  by \cite[2.2.7]{Sp}, closed.  Since ${\rm char}\, k=0$,
for every nonidentity unipotent element $u\in G$,
the closure of
$\{u^n\mid n\in \mathbb Z\}$ is
a one-dimensional unipotent subgroup of $G$
(cf.,\,e.g.,\,\cite[Chap.\,3, \S2, no.\,2, Theorem~1]{OV}).
Hence $G_0$ coincides with the subgroup generated by all unipotent element of $G$.
This yields (ii)$\Leftrightarrow$(iii).

 Since homomorphisms of algebraic groups preserve Jordan decompositions,
 $G_0$ is contained in the kernel of every character of $G$ and
every element of the
$G/G_0$ is semisimple. The latter yields that
$G/G_0$ is a torus (cf.,\,e.g.,\,\cite[I.4.6]{Bo}).
Hence $G$ has no nontrivial characters if and only if $G=G_0$.
This proves (i)$\Leftrightarrow$(ii).

 Since ${\rm char}\,k=0$, there is
 a reductive subgroup $L$ in $G$
 such that $G$ is the semidirect product
 of $\,{\rm Rad}_uG$ and $L$ (cf.,\,e.g.\,\cite[Chap.\,6, Sect.\,4]{OV}).
  Let $Z$ and $Z^0$
  %%, and $(L, L)$
  be resp.\;the center of $L$ and the identity
  component of $Z$. %%and the commutator subgroup of $L$. %%and
   Put $H:=(L,L){\rm Rad}_uG$. Then  $Z^0$ is a torus,
   $F:=(L,L)\cap Z^0$ is finite, $L=Z^0(L, L)$, and
   $H$ is connected and normal. Being connected semisimple,
   $(L, L)$ has no nontrivial characters. Hence
    $H$ is generated by unipotent elements. This yields
    $H\subseteq G_0$. As $G/H$ is isomorphic to $Z^0/F$ and
    the latter is a torus, all elements of $G/H$ are
    semisimple. Hence $H=G_0$. Thus, (i) holds if and
    only if $Z^0$ is the identity. Since
    ${\rm Rad}\,G=Z^0{\rm Rad}_uG$, this proves
  (i)$\Leftrightarrow$(iv).
\end{proof}

\begin{corollary}\label{c11}
${\rm ML}(X)=\bigcap_{H\subseteq {\rm Aut}(X)}
 k[X]^H$, where
$H$ runs over all connected linear algebraic subgroups
of ${\rm Aut}(X)$
that have no nontrivial characters.
\end{corollary}

\begin{theorem} \label{MLsubset}
Let $X$ be a variety and let $\,G$ be
a connected linear algebraic subgroup of ${\rm Aut}(X)$ that
has no nontrivial characters. Then
\begin{equation}\label{MLsub=}
{\rm ML}(X)\subseteq k[X]^G.
\end{equation}
\end{theorem}

\begin{proof}
From Lemma  \ref{GGG} we infer that
$k[X]^G=
 \bigcap_{H} k[X]^H$
 where $H$
 runs over all one-parameter unipotent subgroups of
 $G$.
 This and \eqref{MLL} imply \eqref{MLsub=}.
\end{proof}

\begin{corollary}\label{homomm}
Maintain the notation of Theorem {\rm\ref{MLsubset}}. If $\,G$ has no nontrivial charac\-ters and $k[X]^G=k$, then ${\rm ML}(X)=k$.
\end{corollary}

 Since there are no nonconstant invariant functions on orbit closures, this yields the following.

\begin{corollary}\label{homom}
Maintain the notation of Theorem {\rm\ref{MLsubset}}. If $\,G$ has no nontrivial charac\-ters and $X$ is the closure of a $G$-orbit, then
${\rm ML}(X)=k$.
\end{corollary}

\begin{corollary} \label{G///H}
Let $\,G$ be a connected algebraic group that has no nontrivial characters. Let $H$ be a reductive subgroup of $G$. Then $G/H$ is an irreducible affine variety
with trivial Makar-Limanov invariant.
\end{corollary}
\begin{proof} As $\,G$ acts on $G/H$ transitively,
Corollary \ref{homom} yields ${\rm ML}(G/H)=k$.
By \cite[Theo\-rem 6.8]{Bo} and \cite[The\-orem 4.9]{PV2}
reductivity of $H$ implies that $G/H$ is affine.
\end{proof}

The following
generalizes Corollary \ref{homom}.

\begin{theorem}\label{prop}
Let $X$ be a variety endowed with an action of a connected linear algebraic group $\,G$.  Let
$\,d$ be the dimension of the center of $\,G/{\rm Rad}_uG$. If
$X$ contains a dense $G$-orbit, then
\begin{equation}\label{center}
{\rm tr\,deg}_k {\rm ML}(X)\leqslant d.
\end{equation}
\end{theorem}
\begin{proof} Let $X$ be the closure of the $G$-orbit of a point $x\in X$.
The morphism $G\to X$, $g\mapsto g\cdot x$, is $G$-equivariant
with respect to the action of $G$ on itself by left translations. Since its image is dense in $X$, the corresponding comorphism
 is a $G$-equivariant embedding of the $k$-algebras
  \begin{equation}\label{incl}
 k[X]\hookrightarrow k[G].
  \end{equation}

 Let $L$, $Z$, $Z^0$, $F$, and $H$ be as in the proof of Lemma \ref{GGG}.
 From \eqref{incl} and Theorem \ref{MLsubset} we then infer that
  \begin{equation}\label{incll}
 {\rm ML}(X)\subseteq k[X]^H\hookrightarrow k[G]^H.
  \end{equation}
  Since $G/H$ is isomorphic to $Z^0/F$ and $\dim Z^0/F=\dim Z^0=\dim Z=d$, we have $\dim G/H=d$.
   As $k[G]^H$ is isomorphic to $k[G/H]$, this and  \eqref{incll} imply
     the claim.
\end{proof}

\begin{corollary}\label{cone} Let $X$ be the closure in
${\bf P}^n$
of an orbit of
a connected algebraic subgroup
$G$ in ${\rm Aut}({\bf P}^n)$.
Let
$\,{\widehat X}\subseteq k^{n+1}$ be the affine cone over
$X$. Then ${\rm tr\,deg}_k{\rm ML}({\widehat X})\leqslant d+1$,
where $d$ is the dimension of the center of $\,G/{\rm Rad}_uG$.
\end{corollary}
\begin{proof} Let $\widehat G$ be the pullback of $\,G$ with respect to the natural projection
${\bf GL}_{n+1}\to %%{\bf PGL}_{n+1}=
{\rm Aut}({\bf P}^n)$.
%%${\rm Aut}(k^{n+1})\to{\rm Aut}({\bf P}^n)$.
Then $\widehat X$ is the closure of a $\widehat G$-orbit
in $k^{n+1}$ and the dimension of the center of
$\,{\widehat G}/{\rm Rad}_u{\widehat G}$ is $d+1$; whence the claim by
Theorem \ref{prop}.
\end{proof}

\begin{lemma}\label{otimes} For any varieties $X_1$ and $X_2$,
\begin{equation}\label{tensor}
{\rm ML}(X_1\times X_2)\subseteq {\rm ML}(X_1)\otimes_k {\rm ML}(X_2).
\end{equation}
\end{lemma}

\begin{proof} Take an element $f\in {\rm ML}(X_1\times X_2)$.
Since
$k[X_1\times X_2]$ is generated by $k[X_1]$ and $k[X_2]$,   there
is a decomposition
\begin{equation}\label{deco}
f= \sum_{i=1}^n s_it_i,\quad%%\mbox{where}\quad
s_1,\ldots, s_n\in k[X_1],\;
t_1,\ldots, t_n\in k[X_2].
\end{equation}
We may (and shall) assume that $t_1,\ldots, t_n$ in \eqref{deco}
are linearly independent over $k$. As  $k[X_1\times X_2]=k[X_1]\otimes_k k[X_2]$, then they are
also linearly independent over $k[X_1]$.

Consider an action $\alpha$ of ${\bf G}_a$ on $X_1$.  Then $k[X_1]$ is stable and $k[X_2]$ is pointwise fixed with respect to the diagonal action of ${\bf G}_a$ on $X_1\times X_2$ determined by $\alpha$ and trivial action on $X_2$. For
every element $g\in {\bf G}_a$ and this diagonal action,
\eqref{MLL} and \eqref{deco} imply that
\begin{equation}\label{decom}
 \sum_{i=1}^n s_it_i=f=g\cdot f=\sum_{i=1}^n(g\cdot s_i) t_i.
\end{equation}
Since $t_1,\ldots, t_n$
are linearly independent over $k[X_1]$, we infer from \eqref{decom} that
every $s_i$ is invariant with respect to $\alpha$. As
$\alpha$ is arbitrary, \eqref{MLL} implies that $s_1,\ldots, s_n\in {\rm ML}(X_1)$. Hence $f$ is decomposed~as
\begin{equation}\label{dash}
f= \sum_{i=1}^m s_i' t_i',\quad s_1',\ldots, s_m'\in {\rm ML}(X_1),\; t_1',\ldots, t_m'\in k[X_2],
\end{equation}
where $s_1',\ldots, s_m'$ are linearly independent over $k$.
The same argument as above then yields $t_1',\ldots, t_m'\in {\rm ML}(X_2)$. Now \eqref{tensor}
follows from \eqref{dash}.
\end{proof}

\begin{corollary} \label{c1}
For any varieties $X_1$ and $X_2$, the following %%properties
are equivalent:
\begin{enumerate}
\item[\rm(i)] ${\rm ML}(X_1)$ and $\,{\rm ML}(X_2)$ lie in ${\rm ML}(X_1\times X_2)$;
\item[\rm(ii)]
${\rm ML}(X_1\times X_2)= {\rm ML}(X_1)\otimes_k {\rm ML}(X_2)$.
\end{enumerate}
\end{corollary}

\begin{corollary} %%Let $X_1$ and $X_2$ be algebraic varieties.
 If $\,{\rm ML}(X_1)=k$ and $\,{\rm ML}(X_2)=k$, then ${\rm ML}(X_1\times X_2)=k$.
\end{corollary}

\begin{corollary} Let $X_1$ and $X_2$ be the varieties such that ${\rm ML}(X_1)$ and ${\rm ML}(X_2)$ are generated by units.
 Then
$
{\rm ML}(X_1\times X_2)= {\rm ML}(X_1)\otimes_k {\rm ML}(X_2).
$
\end{corollary}

\begin{proof} This
%%claim
follows from Corollary \ref{c1} since
$k[X_1]^*$ and $k[X_2]^*$ lie in $k[X_1\times X_2]^*$ and
$k[X_1\times X_2]^*\subset {\rm ML}(X_1\times X_2)$, cf.\;\cite[1.4]{F}.
\end{proof}

\begin{definition} \label{def_toral}
A variety is called {\it toral\,} if it is isomorphic to a closed subvariety of a linear algebraic torus.
\end{definition}

Note that closed subvarieties and  products of toral varieties are toral.

\begin{lemma}\label{toral}
Let $X$ be an affine variety.
\begin{enumerate}
\item[\rm(a)] The following
are equivalent:
\begin{enumerate}
\item[$\rm (a_1)$] $X$ is toral;
\item[$\rm (a_2)$] $k[X]$ is generated by  $k[X]^*$.
\end{enumerate}
\item[\rm(b)] For every
finite subgroup $G$ of ${\rm Aut}(X)$, there is a co\-ve\-ring of $X$ by $G$-stable open toral sets.
\item[\rm(c)] If $X$ is toral, then
\begin{enumerate}
\item[$\rm (c_1)$] for every unipotent linear algebraic group $H$, every algebraic homomorphism $\varphi\colon H\to {\rm Aut}(X)$ is trivial{\rm;}
\item[$\rm (c_2)$]  ${\rm ML}(X)=k[X]$.
\end{enumerate}
\end{enumerate}
\end{lemma}
\begin{proof} (a) Every character of a linear algebraic torus $T$
is an element of
 $k[T]^*$ and
   $k[T]^*$ is the $k$-linear span of the set of all characters \cite[Sect.\,8.2]{Bo}; this and Definition  \ref{def_toral} imply $\rm (a_1)$$\Rightarrow$$\rm (a_2)$.

Conversely, if $\rm (a_2)$ holds, let $k[X]=k[f_1,\ldots, f_n]$ for
some $f_i\in k[X]^*$. Then $\iota \colon X\to {\bf A}\!^n$,
$x\mapsto (f_1(x), \ldots, f_n(x))$, is a closed embedding
since $X$ is affine.
The standard coordinate functions on ${\bf A}\!^n$ do not vanish
on $\iota(X)$ since $f_i$ does not vanish on $X$. Hence $\iota(X)\subset ({\bf G}_m)^n$. This proves $\rm (a_2)$$\Rightarrow$$\rm (a_1)$ and completes the proof of (a).

(b) Let $x$ be a point of $X$. We have to show that $x$ is contained in a $G$-stable open toral subset of $X$.
Let $k[X]=k[h_1,\ldots, h_s]$.  Replacing $h_i$ by $h_i+\alpha_i$ for an appropriate $\alpha_i\in k$, we may (and shall) assume that every $h_i$ vanishes nowhere on the $G$-orbit $G\cdot x$ of $x$. Enlarging the set $\{h_1,\ldots, h_s\}$ by including in it $g\cdot h_i$ for every $i$ and $g
\in G$, we may (and shall) assume that
$\{h_1,\ldots, h_s\}$ is $G$-stable. Then $h:=h_1\cdots h_s\in k[X]^G$. Hence the affine open set $X_h:=\{z\in X \mid h(z)\neq 0\}$ is $G$-stable and contains $G\cdot x$.
Since $k[X_h]=k[h_1,\ldots, h_s, 1/h]$
  we have
  $h_i\in k[X_h]^*$ for every $i$. Hence,  $X_h$ is toral by (a). This proves (b).

(c) Consider the action of $H$ on $X$ determined by $\varphi$.
Let $H\cdot x$ be the $H$-orbit of a point $x\in X$.
Since ${\rm char}\,k=0$,
$H\cdot x$ is isomorphic to ${\bf A}\!^d$ for some $d$, see \cite[Cor.\,of Theorem 2]{Po1}. Since $H$ is unipotent and $X$ is affine, $H\cdot x$ is closed in $X$, cf.\,\cite[4.10]{Bo}. Hence $H\cdot x$ is toral. Since $k[{\bf A}\!^d]^*=k^*$, from (a) we then infer that $d=0$, i.e., $x$ is a fixed point. This proves  $\rm (c_1)$. In turn, $\rm (c_1)$ implies $\rm (c_2)$ by \eqref{MLL}.
\end{proof}

\begin{corollary} \label{c2}
If $\,{\rm ML}(X_1)=k$ and $\,X_2$ is toral, then ${\rm ML}(X_1\times X_2)=k[X_2]$.
\end{corollary}

Utilizing the above statements one gets
many interesting varieties with trivial Makar-Limanov invariant.
The following construction is typical (but not the only possible, see Examples \ref{Schubert} and \ref{nonrational}).

Let $G$ be a connected semisimple algebraic group acting on an affine variety $X$. By Hilbert's theorem, $k[X]^G$ is a finitely generated $k$-algebra. Let $k[X]^G=k[f_1,\ldots, f_n]$.
 For every
 $\alpha_1,\ldots,\alpha_n\in k$,
denote by
$X(\alpha_1,\ldots,\alpha_n)$ the closed subvariety of $X$
whose underlying topological space is
$\{x\in X\mid f_1(x)=\alpha_1,\ldots, f_n(x)=\alpha_n\}$
 (warning:
in general, the ideal $(f_1-\alpha_1,\ldots,f_n-\alpha_n)$ of $k[X]$ is not radical).
Let $Y$ be a $G$-stable closed suvariety of $X$. It is well-known
that
$k[X]^G\to k[Y]^G$,\ $f\mapsto f|_Y$,
is an epimorphism \cite[3.4]{PV2}. Hence $k[Y]^G=k$ if and only if $Y$ is contained in some $X(\alpha_1,\ldots,\alpha_n)$. From
Theorem \ref{MLsubset} we then infer that the Makar-Limanov invariant of every
$G$-stable closed subvariety of $X(\alpha_1,\ldots,\alpha_n)$ is trivial.

There are many instances where $f_1,\ldots, f_n$ can be explicitly described. E.g., classical invariant theory yields such a description for a number of finite-dimensional modules  $X$ of classical linear groups $G$; for some of them, it is proved that %%the ideal
$(f_1-\alpha_1,\ldots,f_n-\alpha_n)$ is radical.
If the latter happens, one obtains
 the instances
 of affine algebras
 with trivial Makar-Limanov invariant that are explicitly described by equations.

Below are several
illustrating
examples.

\begin{example}[Closures of adjoint orbits]\label{adj} Let $f_s$ be the sum of all principal $s\times s$-minors of the $n\times n$-matrix $(x_{ij})$ where $x_{11},\ldots, x_{nn}$ are variables considered as the standard coordinate functions
on
${\bf M}_{n\times n}$.
For $\alpha_1,\ldots,\alpha_n\in k$,
\begin{equation*}
{\bf M}_{n\times n}(\alpha_1,\ldots,\alpha_n):=\{a\in {\bf M}_{n\times n} \mid f_1(a)=\alpha_1,\ldots, f_n(a)=\alpha_n\}
\end{equation*}
is the set of all matrices
whose characteristic polynomial
is $t^n+\sum_{i=1}^n(-1)^i\alpha_it^{n-i}$.

Consider the action of ${\bf SL}_n$ on ${\bf M}_{n\times n}$ by conjugation.
 Then  $k[{\bf M}_{n\times n}]^{{\bf SL}_n}$ is freely generated by $f_1,\ldots, f_n$ (cf.,\;e.g.,\;\cite[0.6]{PV2}). Moreover,
 ${\bf M}_{n\times n}(\alpha_1,\ldots,\alpha_n)$
  is irreducible and
 the ideal
  $(f_1-\alpha_1,\ldots, f_n-\alpha_n)$ of $k[{\bf M}_{n\times n}]$
is
radical
 (see the next paragraph). Hence,
 ${\bf M}_{n\times n}(\alpha_1,\ldots,\alpha_n)$
   is a closed subvariety of ${\bf M}_{n\times n}$ such that
\begin{equation*}
{\rm ML}(
{\bf M}_{n\times n}(\alpha_1,\ldots, \alpha_n))
=k
\end{equation*}
and $k[\ldots, x_{ij},\ldots]/(f_1-\alpha_1,\ldots,f_n-\alpha_n)$ is the $k$-domain with trivial Makar-Limanov invariant.

This admits the following generalization.
Let $G$ be a connected reductive algebraic group and let ${\rm Lie}(G)$ be the Lie algebra of $G$ endowed with the adjoint action of $G$. By \cite{K} the graded
$k$-algebra $k[{\rm Lie}(G)]^G$ is free and, for every minimal system of its homogeneous
generators $f_1,\ldots, f_{r}$ and constants $\alpha_1,\ldots,\alpha_r\in k$,
\begin{enumerate}
\item[(i)]
 ${\rm Lie}(G){(\alpha_1,\ldots,\alpha_r)}:=\{a\in {\rm Lie}(G)\mid f_1(a)=\alpha_1,\ldots, f_r(a)=\alpha_r\}$
is the closure of a $G$-orbit;
\item[(ii)] the ideal $(f_1-\alpha_1,\ldots, f_r-\alpha_r)$ of  $k[{\rm Lie}(G)]$ is radical.
    \end{enumerate}
 Since the center $Z$ of $G$ acts trivially on ${\rm Lie}(G)$ and $G/Z$ is semisimple, this yields
 \begin{equation*}
 {\rm ML}({\rm Lie}(G){(\alpha_1,\ldots,\alpha_r)})=k
 \end{equation*}
 and $k[{\rm Lie}(G)]/(f_1\!-\!\alpha_1,\ldots, f_r\!-\!\alpha_r)$ is the $k$-domain with
  trivial Makar-Limanov invariant.

 For $G={\bf GL}_n$, %%this yields the varieties
 we have
 ${\rm Lie}(G){(\alpha_1,\ldots,\alpha_r)}={\bf M}_{n\times n}(\alpha_1,\ldots,\alpha_n)$. %%
  \end{example}

\begin{example}[Determinantal varieties]\label{determ} Given positive integers $n\geqslant m>r$, let $\{x_{ij} \mid i=1,\ldots, n,\, j=1,\ldots, m\}$ be
 the set of variables
 considered  as the standard coordinates functions on
 ${\bf M}_{n\times m}$. Let $I_{n,m,r}$ be
the ideal  of $k[{\bf M}_{n\times m}]=k[\ldots, x_{ij},\ldots]$
generated by all $(r+1)\times (r+1)$-minors of the matrix
$(x_{ij})$.
Then $I_{n,m,r}$ is radical, cf., e.g., \cite{Pr}. The (affine) determinantal variety
$D_{n,m,r}$ is the subvariety of
${\bf M}_{n\times m}$
 defined by $I_{n,m,r}$. Its underlying set is that of $n\times m$-matrices of rank $\leqslant r$. It is stable with respect to the action of
${\bf SL}_n\times {\bf SL}_m$ on ${\bf M}_{n\times m}$ by $(g, h)\cdot a:=gah^{-1}$ and contains a dense orbit. Whence
\begin{equation*}
{\rm ML}(D_{n,m,r})=k
\end{equation*}
and $k[{\bf M}_{n\times m}]/I_{n,m,r}$ is the $k$-domain with trivial Makar-Limanov invariant.
\end{example}

\begin{example}[$S$-varieties in the sense of \cite{PV1}]\label{S} Denote by ${\sf S}^d{k}^n$  the  $d$th symmetric power of the coordinate vector space (of columns) ${k}^n$. The natural ${\bf SL}_n$-action on
${k}^n$ induces
 that on ${\sf S}^d{k}^n$. The (affine) Veronese morphism
\begin{equation*}
\nu_{n}^{d}\colon {k}^n\to {\sf S}^d{k}^n, \qquad v\mapsto v^d,
\end{equation*}
is ${\bf SL}_n$-equivariant.
Its image $\nu_{n}^{d}({k}^n)$
is closed and contains a dense
${\bf SL}_n$-orbit. The ideal of
$\nu_{n}^{d}({k}^n)$ is generated by all $2\times 2$-minors
of a certain symmetric matrix whose entries are the
coordinates on ${\sf S}^d{k}^n$, cf.\;\cite{Ha}. Thus,
\begin{equation*}
{\rm ML}\big(\nu_{n}^{d}({k}^n)\big)=k
\end{equation*}
and
the coordinate algebra of $\nu_{n}^{d}({k}^n)$ is the explicitly described $k$-domain with trivial Makar-Limanov invariant.

More generally,
the following combination of the Veronese and Segre morphisms
\begin{gather*}
\nu_{n_1,\ldots,n_s}^{d_1,\ldots,d_s}\colon
{k}^{n_1}\times\cdots\times {k}^{n_s}\to {\sf S}^{d_1}{k}^{n_1}\otimes\cdots\otimes {\sf S}^{d_s}{k}^{n_s},\\
(v_1,\ldots, v_s)\mapsto v^{d_1}\otimes\cdots\otimes v^{d_s},
\end{gather*}
is equivariant with respect to the natural
${\bf SL}_{n_1}\times\cdots\times {\bf SL}_{n_s}$-actions, its image is closed and contains a dense
orbit. Whence,
\begin{equation*}
{\rm ML}\big(\nu_{n_1,\ldots,n_s}^{d_1,\ldots,d_s}({k}^{n_1}\times
\cdots\times {k}^{n_s})\big)=k.
\end{equation*}

 In turn, this construction admits a further generalization. Namely, any matrix  $$A=\begin{pmatrix}a_{11}\ldots
a_{1s}\\
\ldots\ldots\ldots\\
a_{r1}\ldots a_{rs}
\end{pmatrix}$$ with the entries in ${\bf Z}_{>0}$ defines the diagonal morphism
\begin{equation*}
\nu_{n_1,\ldots, n_s}^A:=\nu_{n_1,\ldots,
n_s}^{a_{11},\ldots, a_{1s}}\times\cdots\times
\nu_{n_1,\ldots, n_s}^{a_{r1},\ldots, a_{rs}}.
\end{equation*}
 This morphism is  ${\bf SL}_{n_1}\times\cdots\times {\bf SL}_{n_s}$-equivariant and its image
\begin{equation*}
{\bf H}_{n_1,\ldots, n_s}^A:=\nu_{n_1,\ldots, n_s}^A({k}^{n_1}\times\cdots\times {k}^{n_s})
\end{equation*}
is closed and contains a dense %%${\bf SL}_{n_1}\times\cdots\times {\bf SL}_{n_s}$-
orbit. Thus,
\begin{equation*}
{\rm ML}({\bf H}_{n_1,\ldots, n_s}^A)=k.
\end{equation*}

In fact, ${\bf H}_{n_1,\ldots, n_s}^A$'s are special examples of varieties with trivial Makar-Limanov invariant obtained by the following general construction \cite{PV1}.

Let $G$ be a connected semisimple algebraic group and let $E(\lambda)$ be a simple $G$-module with the highest weight
$\lambda$ (with respect to a fixed Borel subgroup and its maximal torus). Let $v_{\lambda_i}$ be a highest vector of $E(\lambda_i)$.
For $x=v_{\lambda_1}+\cdots +v_{\lambda_s}\in E(\lambda_1)\oplus\cdots\oplus E(\lambda_s)$,
let
$X(\lambda_1,\ldots,\lambda_s)$ be the closure of the $G$-orbit of $x$. Up to $G$-isomorphism,
$X(\lambda_1,\ldots,\lambda_s)$ depends
only on $\lambda_1,\ldots,\lambda_s$. By Corollary~\ref{homom}
\begin{equation*}
{\rm ML}(X(\lambda_1,\ldots,\lambda_s))=k.
\end{equation*}

The ideal $I(\lambda)$ of  $X(\lambda)$ in $k[E(\lambda)]$ is generated by quadratic forms that can be explicitly described. Namely, $k[E(\lambda)]$ is the symmetric algebra of
the dual $G$-module $E(\lambda)^*=E(\lambda^*)$. The submodule
${\sf S}^2E(\lambda^*)$ of the $G$-module $k[E(\lambda)]$
contains a unique simple submodule with the highest weight
$2\lambda^*$, the Cartan component of ${\sf S}^2E(\lambda^*)$.
Hence
${\sf S}^2E(\lambda^*)$ contains a unique
submodule $L$ complement to the Cartan component.
 This $L$
generates $I(\lambda)$, see \cite[Theorem 4.1]{Br}.

For
$G={\bf SL}_{n_1}\times\cdots\times {\bf SL}_{n_s}$
and $\lambda_i=a_{i1}\varpi_1^{(1)}+\cdots +a_{is}\varpi_1^{(s)}$ where $\varpi_1^{(j)}$ is
the highest weight of the natural ${\bf SL}_{n_j}$-module
${k}^{n_j}$, we have
$$X(\lambda_1,\ldots,\lambda_s)={\bf H}^A_{n_1,\ldots,n_s}.$$
\end{example}

\begin{example}[Irreducible affine surfaces quasihomogeneous  with respect to an algebraic group in the sense of \cite{G}]\label{surf}
By \cite{Po1},
up to isomorphism,
such surfaces are exhausted  by the following list
(we maintain the notation of Example \ref{S}):
\begin{enumerate}
\item[(i)] smooth surfaces:
\begin{equation}\label{smooo}
{\bf A}\!^2,\; {\bf A}\!^1\times {\bf A}\!^1_*,\; {\bf A}\!^1_*\times {\bf A}\!^1_*,\; ({\bf P}^1\times {\bf P}^1)\setminus\Delta,\;
{\bf P}^2\setminus C,
\end{equation}
where  $\Delta$ is the diagonal in ${\bf P}^1\times {\bf P}^1$, and
$C$ is a nondegenerate conic in ${\bf P}^2$;
\item[(ii)] singular surfaces:
$${\bf V}(n_1,\ldots,n_r):={\bf H}_2^{A}\qquad \mbox{for $A=
(n_1,\ldots,n_r)^{\sf T}$, \quad $n_1,\ldots, n_r\geqslant 2$}.$$
\end{enumerate}

 Each of these surfaces but ${\bf A}\!^1\times {\bf A}\!^1_*$ and ${\bf A}\!^1_*\times {\bf A}\!^1_*$ admits an
${\bf SL}_2$-action with a dense orbit. Namely,  $({\bf P}^1\times {\bf P}^1)\setminus\Delta={\bf SL}_2/T$ and ${\bf P}^2\setminus C={\bf SL}_2/N(T)$, where $T$ is a maximal torus of ${\bf SL}_2$ and $N(T)$ its normalizer, see \cite[Lemma 2]{Po1}. The surface ${\bf V}(n_1,\ldots, n_r)$ is the closure of the ${\bf SL}_2$-orbit of
$v_1+\cdots+v_r\in {\bf R}_{n_1}\oplus\cdots\oplus
{\bf R}_{n_r}$, where $v_i$ is a highest vector of the simple
${\bf SL}_2$-module ${\bf R}_{n_i}$ of dimension $n_i+1$ (such a module is unique up to isomorphism),
see \cite[\S2]{Po1}. By
Corollary \ref{homom} this yields
\begin{equation}\label{dc}
{\rm ML}(({\bf P}^1\times {\bf P}^1)\setminus\Delta)={\rm ML}
({\bf P}^2\setminus C)={\rm ML}({\bf V}(n_1,\ldots,n_r))=k,
\end{equation}

As ${\bf A}_*^m:=({\bf A}_*^1)^m$ is toral and ${\rm ML}({\bf A}\!^n)=k$,
Corollary \ref{c2} implies that
\begin{equation}\label{=}
{\rm ML}({\bf A}\!^n\times {\bf A}_*^m)=k[{\bf A}_*^m].
\end{equation}
From \eqref{=} we get the Makar-Limanov invariants of the remaining three surfaces in \eqref{smooo}.
\end{example}

\begin{example}[Irreducible affine threefolds quasihomogeneous  with respect to an algebraic group in the sense of \cite{G}]
\label{threef}
We maintain the notation of Examples \ref{S} and \ref{surf}.
Identify ${\rm Pic}(({\bf P}^1\times {\bf
P}^1)\setminus \Delta)$ with ${\bf Z}$  by a fixed
isomorphism $\varphi$.
Let ${\bf X}_n$ be the total space of the
one-dimensional vector bundle over $({\bf P}^1\times
{\bf P}^1)\setminus \Delta$ corresponding to $n\in
{\bf Z}$, and let ${\bf X}_n^*$ be the complement of
the zero section in ${\bf X}_n$.
In fact,
${\bf X}_n$ is isomorphic to ${\bf X}_{-n}$ and ${\bf
X}_n^*$ to ${\bf X}_{-n}^*$, so
${\bf
X}_n$ and ${\bf X}_n^*$ do not depend on the choice
of $\varphi$, see \cite{Po2}.

The group ${\rm Pic}\big({\bf P}^2\setminus C\big)$ has
order $2$. Let ${\bf Y}_0$ and ${\bf Y}_1$ be the
total spaces of, resp., trivial and
nontrivial one-dimensional vector bundles over ${\bf
P}^2\setminus C$. Let ${\bf Y}_n^*$ be the
complement of the zero section in ${\bf Y}_n$.

Let  $\widetilde {\rm T}$, $\widetilde {\rm O}$, $\widetilde {\rm I}$,
and ${\widetilde {\rm D}}_n$ be, resp., the binary
 tetrahedral, octahedral,
icosahedral, and dihedral subgroup of order $4n$ in
${\bf SL}_2$. Put ${\bf S}_3={\bf
SL}_2/\widetilde {\rm T}$, ${\bf S}_4={\bf
SL}_2/\widetilde {\rm O}$, ${\bf S}_5={\bf
SL}_2/\widetilde {\rm I}$, and ${\bf W}_n={\bf
SL}_2/{\widetilde {\rm D}}_n$.

By \cite{Po1}
up to isomorphism irreducible
affine threefolds quasihomogeneous with respect to an algebraic group in the sense of \cite{G}
are exhausted  by the following list:
\begin{enumerate}
\item[\rm(i)] smooth threefolds:
\begin{gather}\label{smooth}
\begin{gathered}
{\bf X}_n,\;
{\bf X}_n^*,\;
{\bf W}_n,\;
{\bf Y}_0,\;
{\bf Y}_0^*,\;
{\bf Y}_1^*,\;
{\bf S}_3,\;
{\bf S}_4,\;
{\bf S}_5,\\
{\bf A}\!^3,\;
{\bf A}\!^2\times {\bf A}_*^1,\;
{\bf A}\!^1\times {\bf A}_*^2
,\;
{\bf A}_*^3;
\end{gathered}
\end{gather}
\item[\rm(ii)] singular threefolds:
\begin{gather*}
{\bf P}(A):={\bf H}_3^A\quad\mbox{where all entries of $A$ are $\geqslant 1$},\\
{\bf Q}(B):={\bf H}_{2,2}^B\quad \mbox{where $\;{\rm rk}\,B=1$}.
\end{gather*}
\end{enumerate}

By construction, ${\bf S}_3$, ${\bf S}_4$, ${\bf S}_5$, ${\bf W}_n$ are homogeneous with respect to ${\bf SL}_2$ while  ${\bf P}(A)$ and ${\bf Q}(B)$ admit an action of resp. ${\bf SL}_2\times {\bf SL}_2$ and ${\bf SL}_3$ with a dense orbit. In fact, ${\bf X}_n^*$ for $n\neq 0$ is homogeneous with respect to ${\bf SL}_2$ as well (it is the quotient of ${\bf SL}_2$ modulo a cyclic subgroup of order $|n|$). By \cite[Theorem 9]{Po2} every ${\bf X}_n$  is homogeneous with respect to the nonreductive linear
algebraic group ${\bf SL}_{2,|n|}:={\bf SL}_2\ltimes {\bf R}_{|n|}$ (see Example \ref{surf});
the radical of ${\bf SL}_{2,|n|}$ is unipotent. By \cite[Prop.\,18]{Po2}, ${\bf Y}_0$ is homogeneous with respect to
${\bf SL}_{2, d}$ for every even $d>0$.
From
Theorem \ref{prop} we then deduce that
\begin{gather*}
{\rm ML}({\bf S}_3)={\rm ML}({\bf S}_4)={\rm ML}({\bf S}_5)={\rm ML}({\bf Y}_0)=k,\\
{\rm ML}({\bf X}_n)={\rm ML}({\bf W}_n)={\rm ML}({\bf P}(A))
={\rm ML}({\bf Q}(B))=k,\\
{\rm ML}({\bf X}_n^*)=k\quad\mbox{for}\hskip 2mm n\neq 0.
\end{gather*}

As ${\bf X}_0^*=(({\bf P}^1\times {\bf P}^1)\setminus \Delta)\times {\bf A}_*^1$ and ${\bf Y}_0^*=({\bf P}^2\setminus C)\times {\bf A}_*^1$, we deduce from \eqref{dc} and Corollary \ref{c2} that
\begin{equation*}
{\rm ML}({\bf X}_0^*)= {\rm ML}({\bf Y}_0^*)=
k[{\bf A}_*^1].
\end{equation*}

 By \cite[Prop.\,16]{Po2},
 ${\bf Y}_1^*$ is homogeneous with respect to  ${\bf SL}_2\times {\bf G}_m$. Since the latter is a reductive group with  one-dimensional center, Theorem \ref{prop} implies that
 ${\rm tr\,deg}_k{\rm ML}({\bf Y}_1^*)\leqslant 1$. On the other hand, by \cite[Prop.\,19]{Po2}, $k[{\bf Y}_1^*]/k^*$ is a free abelian group of rank $1$. Since $k[X]^*\subseteq {\rm ML}(X)$ for every $X$, this yields
 \begin{equation*}
 {\rm tr\,deg}_k{\rm ML}({\bf Y}_1^*)=1.
 \end{equation*}

 Finally, \eqref{=} yields the Makar-Limanov invariants of the last four threefolds in \eqref{smooth}.
\end{example}

\begin{example}[Schubert varieties] \label{Schubert}
Let $G$ be a connected semisimple algebraic group and let ${\bf P}E$ be the projective space of $1$-dimensional linear subspaces in a nonzero simple
$G$-module $E$. There is a unique closed $G$-orbit $\mathcal O$ in ${\bf P}E$. Let $U$ be a maximal unipotent subgroup of $G$.
There are only finitely many $U$-orbits in $\mathcal O$;  their closures are called Schubert varieties, cf.,\,e.g.,\,\cite[8.3--8.5]{Sp}. Let $X\subseteq \mathcal O$ be a Schubert variety and let $\widehat X$ be the affine cone  over $X$ in $E$.
As $U$ is unipotent, Corollary \ref{cone} yields
\begin{equation*}
{\rm tr\,deg}_k{\rm ML}({\widehat X})\leqslant 1.
\end{equation*}
The ideal of $\widehat X$ in $k[E]$ is generated by certain forms of degree $\leqslant 2$ that admit an explicit description,
see, e.g.,\;\cite[2.10]{BL}.
\end{example}

\begin{example}[Not stably rational smooth affine varieties with trivial Makar\-Limanov invariant] \label{nonrational}
In \cite{Ln}  a construction of nonrational singular affine threefolds with tri\-vial Makar-\-Limanov invariant is exhibited.~Our approach yields, for every   integer $d$, examples of
%%nonrational\,---\,and even
not stably rational
(hence a fortiori nonrational)
%%\,---\,
smooth affine varieties of dimension $\geqslant d$ with trivial Makar-Limanov invariant. Here is the construction.

Let $F$ be a linear algebraic group. By \cite[Theorem 1.5.5]{Po4} the following pro\-per\-ties are equivalent:
\begin{enumerate}
\item[(a)] for every locally free  finite-dimensional algebraic $kF$-module $V$, the field $k(V)^F$ is stably rational over $k$;
    \item[(b)] there is an embedding $F\subseteq H$, where $H$ is a special group in the sense of Serre, such that
        the variety $H/F$ is stably rational.
\end{enumerate}
Here ``locally free'' means that $F$-stabilizers
of points in general position in $V$ are trivial (see \cite[1.2.2]{Po4}); for finite $F$, this is equivalent to triviality of the kernel of action.
About special groups see, e.g.,\;\cite[1.4]{Po4}, \cite[2.6]{PV2}.

Now let $F$ be a finite group  whose
Schur multiplier $H^2(F, {\bf Q}/{\bf Z})$ contains
a nonzero element $\alpha$ such that $\alpha|_{A}=0$ for
every abelian subgroup $A$ of $F$.
  It is known that such groups exist and, for every locally free finite-dimensional algebraic $kF$-module $V$, the field $k(V)^F$ is not stably rational over $k$ (see, e.g.,\;\cite{Sh}).  By  \cite[2.3.7]{Sp} we can (and shall) embed $F$ in ${\bf SL}_n$ for some $n$.~As ${\bf SL}_n$ is special
  (cf.\;\cite[1.4]{Po4}, \cite[2.6]{PV2}), the aforesaid yields that the smooth variety $X:={\bf SL}_n/F$ is not stably rational. As $F$ is reductive and ${\bf SL}_n$ is a connected algebraic group that has no nontrivial characters, Corollary~\ref{G///H}
  implies that $X$ is a not stably rational smooth affine variety with trivial Makar-Limanov invariant.
  \end{example}

\subsection{The Derksen invariant}\

Let $X$ be a variety.
Recall that the {\it Derksen invariant} $\,{\rm D}(X)$ of $X$ is the $k$-subalgebra  of $k[X]$ generated by all
$k[X]^{H}$'s where $H$
 runs over all
 subgroups of ${\rm Aut}(X)$ isomorphic to ${\bf G}_a$.
 If there are no such subgroups, we put ${\rm D}(X)=\varnothing$.

\begin{example} If $X$ is toral, then ${\rm D}(X)=\varnothing$
by Lemma \ref{toral}$({\rm c}_1)$.
\end{example}

In this section we deduce some information on ${\rm D}(X)$  in case when ${\rm Aut}(X)$ contains a connected noncommutative reductive algebraic subgroup.

Recall that if an algebraic group $G$ acts linearly on  a
(not necessarily finite-dimen\-si\-onal) $k$-vector space $V$, then the $G$-module $V$ is called {\it algebraic} if every element of $V$ is contained in an algebraic finite-dimensional $G$-submodule of $V$,
cf.,\;e.g.,\;\cite[3.4]{PV2}.

The starting point is
\begin{lemma}\label{unip}
Let $\,G$ be a connected noncommutative reductive algebraic group. Then
every algebraic $G$-module $V$ is a $k$-linear span of the set
\begin{equation}\label{unippo}
 \bigcup_{H\subset G} V^H,
\end{equation}
 where $H$ in \eqref{unippo} runs over all one-parameter unipotent subgroups of
 $\,G$.
\end{lemma}
\begin{proof} The assumptions that $G$ is reductive, ${\rm char}\,k=0$, and
$V$ is algebraic imply that $V$ is a sum of
simple $G$-submodules. Hence we may (and shall) assume that $V$ is a nonzero
simple $G$-module. Since $G$ is
a connected noncommutative reductive algebraic group, it contains a one-dimensional unipotent
subgroup $U$ (indeed, since
$(G, G)$
is a nontrivial semisimple group,
a root subgroup of $(G, G)$ with
respect to a maximal torus may be taken as $U$). By the Lie--Kolchin theorem $V^U\neq \{0\}$. Let $v$ be a nonzero vector of $V^U$. As $g\cdot v\in V^{gUg^{-1}}$  for every element $g\in G$, the $G$-orbit $G\cdot v$ of $v$ is contained in set \eqref{unippo}. Hence the $k$-linear span of $G\cdot v$ is contained in the $k$-linear span of this set.
But the $k$-linear span of $G\cdot v$ is $G$-stable and therefore coincides with $V$ since $V$  is simple.
 This completes the proof.
\end{proof}

\begin{theorem}\label{Dtriv}
 Let $X$ be a variety. If $\,{\rm Aut}(X)$ contains
a connected noncommutative reductive algebraic subgroup, then
\begin{equation}\label{DI}
{\rm D}(X)=k[X].
\end{equation}
\end{theorem}
\begin{proof} Let $G$ be a connected noncommutative reductive algebraic subgroup of ${\rm Aut}(X)$. Since the $G$-module $k[X]$ is
algebraic (see\,\cite[Lemma 1.4]{PV2}), the claim
follows from Lemma
\ref{unip} and the definition on ${\rm D}(X)$.
\end{proof}

\begin{remark}
{\rm The following
are equivalent:
\begin{enumerate}
\item[(i)] $\,{\rm Aut}(X)$ contains
a connected noncommutative reductive algebraic subgroup;
\item[\rm(ii)] $\,{\rm Aut}(X)$ contains ${\rm SL}_2$ or
${\rm PSL}_2$.
\end{enumerate}

Indeed, ${\rm SL}_2$ and ${\rm PSL}_2$ are connected noncommutative reductive algebraic groups and every connected noncommutative reductive algebraic group contains ${\rm SL}_2$ or ${\rm PSL}_2$, cf.\;\cite[Theorem 13.18(4)]{Bo}, \cite[7.2.4]{Sp}}.
\end{remark}

The following example shows that the assumption of noncommutativity
in Theorem \ref{Dtriv} cannot be dropped.

\begin{example}\label{KRT}
By \cite{Der}, for the Koras--Russell cubic threefold $X$, the following inequality distinguishing $X$ from ${\bf A}\!^3$ holds:
\begin{equation}\label{neq}
{\rm D}(X)\neq k[X].
\end{equation}
On the other hand, since $X$
is defined
in ${\bf A}\!^4$
by
$x_1+x_1^2x_2+x_3^2+x_4^3=0$
where $x_1,\ldots, x_4$ are the standard
coordinate functions on ${\bf A}\!^4$,
it is stable with respect to the action of
$\,{\bf G}_m$ on ${\bf A}\!^4$ defined by
$t\cdot (a_1, a_2, a_3, a_4)
=(t^6a_1, t^{-6}a_2, t^3a_3, t^2a_4)$.
  Hence ${\rm Aut}(X)$ contains a one-dimensional connected commutative reductive subgroup, cf.\;\cite[Sect.\,3]{DM-JP}.  \end{example}

 One can apply Theorem \ref{Dtriv} to proving that, for some varieties $X$, there are no connected noncommutative reductive algebraic subgroups in $\,{\rm Aut}(X)$.

\begin{example} For the Koras--Russell cubic threefold $X$,
Theorem \ref{Dtriv} and \eqref{neq} imply that $\,{\rm Aut}(X)$ contains no connected noncommutative reductive algebraic subgroups. %%(a description of $\,{\rm Aut}(X)$ is given in
\end{example}

 Since
${\rm Aut}({\bf A}\!^n)$ for $n\geqslant 2$ contains a connected noncommutative reductive algebraic subgroup (for instance, ${\bf GL}_n$), the
  next corollary
   generalizes
    the well-known fact  that
${\rm D}(X\times {\bf A}\!^n)=k[X\times {\bf A}\!^n]$
for $n\geqslant 2$
(see, e.g.,\;\cite{CM}).

\begin{corollary} \label{productt}
Let $Z$ be a
variety such that ${\rm Aut}(Z)$ contains a connected noncommutative reductive algebraic subgroup. Then, for every
variety $X$,
\begin{equation}\label{ttimes}
{\rm D}(X\times Z)=k[X\times Z].
\end{equation}
\end{corollary}
\begin{proof} Consider the natural action of
${\rm Aut}(Z)$ on $Z$ and its trivial action on $X$.
Then the diagonal action of ${\rm Aut}(Z)$ on $X\times Z$
identifies ${\rm Aut}(Z)$ with a subgroup
of ${\rm Aut}(X\times Z)$. Whence the claim by Theorem \ref{Dtriv}.
\end{proof}

The following example shows that
the assumption of noncommutativity in Corollary \ref{productt}
cannot be dropped.

\begin{example} \label{aa}
Let $x_1, x_2$ be the standard coordinate functions on ${\bf A}\!^2$.
The principal open set $Y$
in ${\bf A}\!^2$ defined by $x_1\neq 0$ is isomorphic
to ${\bf A}\!^1_*\times {\bf A}\!^1$
and
\begin{equation}\label{ts}
k[Y]=k[t, t^{-1}, s], \qquad\mbox{where}\quad t:=x_1|_Y,\; s:=x_2|_Y.
\end{equation}

Since
$t$
is the unit of $k[Y]$, for every
action of ${\bf G}_a$ on $Y$ we have
\begin{equation}\label{orbiii}
t, t^{-1}\in k[Y]^{{\bf G}_a}.
\end{equation}
As, clearly, ${\rm Aut}(Y)$ contains a one-dimensional unipotent subgroup, \eqref{orbiii} and the de\-fi\-nition of ${\rm D}(Y)$ yield
 $k[t, t^{-1}]\subseteq {\rm D}(Y)$.
  We also deduce from \eqref{orbiii} that,
 for every point $y\in Y$, the ${\bf G}_a$-orbit
 of $y$ lies in the line
  defined by the equation $t=t(y)$.
 But this orbit
is closed in $Y$ since $Y$ is affine and ${\bf G}_a$ is unipotent, cf.\;\cite[4.10]{Bo}. Hence,
if $y$ is not a fixed point, this orbit coincides with the aforementioned line.
Therefore, if ${\bf G}_a$ acts on $Y$ nontrivially,
$t$ separates orbits in general position. Since ${\rm char}\,k=0$, by
\cite[Lemma 2.1]{PV2} this means that
$k(Y)^{{\bf G}_a}=k(t)$. Whence by \eqref{ts} we have $k[Y]^{{\bf G}_a}=k[t, t^{-1}]$. From this, \eqref{=} and \eqref{ts} we then infer that
\begin{equation*}
k[t, t^{-1}]={\rm ML}({\bf A}\!^1_*\times {\bf A}\!^1)={\rm D}({\bf A}\!^1_*\times {\bf A}\!^1)\varsubsetneq k[{\bf A}\!^1_*\times {\bf A}\!^1]=k[t, t^{-1}, s].
\end{equation*}

Thus, \eqref{ttimes}
 does not hold for $X={\bf A}\!^1_*$, $Z={\bf A}\!^1$ while both ${\rm Aut}({\bf A}\!^1_*)$ and
${\rm Aut}({\bf A}\!^1)$
contain a one-dimensional connected commutative reductive algebraic subgroup.
\end{example}

\begin{theorem} \label{times}
If $\,X_i$ is a variety
such that
$\,{\rm ML}(X_i)\neq k[X_i]$, $i=1, 2$,
then
$$
{\rm D}(X_1\times X_2)=k[X_1\times X_2].
$$
\end{theorem}
\begin{proof}
As  $\,{\rm ML}(X_1)\neq k[X_1]$, there is a nontrivial
${\bf G}_a$-action $\alpha$ on $X_1$. The diagonal ${\bf G}_a$-action on $X_1\times X_2$ determined by $\alpha$ and trivial action on $X_2$
is a nontrivial ${\bf G}_a$-action for which
$k[X_2]$ lies in the algebra of invariants.
Hence, $k[X_2]\subseteq {\rm D}(X_1\times X_2)$. Similarly,
$k[X_1]\subseteq {\rm D}(X_1\times X_2)$. As $k[X_1\times X_2]$ is ge\-ne\-rated by $k[X_1]$ and  $k[X_2]$, the claim follows.
\end{proof}

\begin{example} If $X$ is the  Koras--Russell cubic threefold $X$, then
${\rm D}(X)\neq k[X]$ by \cite{Der}. But for the square of $X$ we have ${\rm D}(X\times X)=k[X\times X]$\,---\,
 since ${\rm ML}(X)=k[x_1|^{\ }_X]\neq k[X]$ (cf.,\;e.g.,\;\cite[Chap.\,9]{F}),
this follows from Theorem
\ref{times}.

\end{example}

\subsection{Generalizations}\

 The
Makar-Limanov and Derksen invariants can be naturally generalized.

Namely, let $X$ be a variety and let $F$ be an algebraic group.

\begin{definition}\label{Kerrrrr}
The {\it $F$-kernel\,} of $X$ is the following $k$-subalgebra of $k[X]$:
\begin{equation}\label{Ker}
{\rm Ker}_F^{}(X):= \bigcap_{H}k[X]^{H},
\end{equation}
where $H$ in \eqref{Ker} runs over the images of all algebraic homomorphisms
$F\to {\rm Aut}(X)$.
\end{definition}
\begin{definition} \label{Envvv} The $F$-{\it envelope} of $X$ is the $k$-subalgebra $${\rm Env}^{}_F(X)$$ of $k[X]$ generated by all $k[X]^H$'s where $H$ runs over all subgroups of ${\rm Aut}(X)$
isomorphic to $F$. If there are no such subgroups, we put ${\rm Env}^{}_F(X)=\varnothing$.
\end{definition}

\begin{example} %%Definitions \ref{Kerrrrr} and \ref{Envvv}
 The definitions imply that
\begin{equation*}
\qquad
{\rm Ker}_{{\bf G}_a}^{}(X)={\rm ML}(X),
\qquad
{\rm Env}_{{\bf G}_a}^{}(X)={\rm D}(X).
\end{equation*}
\end{example}

\begin{definition} \label{FFF} We say that an algebraic group $G$ is $F$-{\it generated} if it is generated by the images of all homomorphisms $F\to G$.
\end{definition}

\begin{examples}\label{Gm} (1) By Lemma \ref{GGG} a connected linear algebraic group $G$ is ${\bf G}_a$-generated if and only if $G$ has no nontrivial characters that, in turn, is equivalent to the condition ${\rm Rad}\,G= {\rm Rad}_uG$.

(2) Every connected reductive algebraic group $G$ is ${\bf G}_m$-generated. This is clear if $G$ is a torus.
%%In the
The general case
%%this
follows from the case of torus because of the following two facts:
%%that
 (a) the subgroup of $G$ generated by connected algebraic subgroups is closed (see,\,e.g.,\,\cite[2.2.7]{Sp}); and (b) the union of maximal tori of $G$ contains a dense open subset of $G$ (\cite[6.4.5(iii), 7.6.4(ii)]{Sp}).

(3) Clearly, the subgroup generated by the images of all homomorphisms $F\to G$ is normal. Hence, if
 $G$ is simple as abstract group and
 there exists a nontrivial homomorphism $F\to G$, then $G$ is $F$-generated.
\end{examples}

The following are the generalizations of the above statements on ${\rm ML}(X)$ and ${\rm D}(X)$.

\begin{theorem}\label{gFsubset} If a variety $X$ is endowed with an action of
an $F$-generated algebraic
group $G$, then
$
{\rm Ker}^{}_F(X)\subseteq k[X]^G.
$
\end{theorem}
\begin{proof}
This follows from Definitions \ref{Kerrrrr} and \ref{FFF}.
\end{proof}

\begin{corollary}\label{FGFG}
If a variety $X$ is endowed with an action of
an $F$-generated algebraic
group $G$ and $X$ contains a dense $G$-orbit, then
${\rm Ker}^{}_F(X)=k$.
\end{corollary}

\begin{corollary}
If $H$ is a reductive subgroup of an $F$-generated linear algebraic group $G$, then $G/H$ is an affine variety with
$
{\rm Ker}^{}_F(G/H)=k.
$
\end{corollary}

\begin{corollary}\label{conenen} Let $X$ be an irreducible variety. If there is an action of $\,{\bf G}_m$ on $X$ with a fixed point and without other closed orbits, then
%%$X$ is affine and
\begin{equation}\label{GmGmGm}
{\rm Ker}^{}_{{\bf G}_m}(X)=k.
\end{equation}
%%If, moreover, $X$ is normal, then $X$ is affine.
\end{corollary}
\begin{proof} The assumptions imply that the fixed point is unique and lies in the closure of every ${\bf G}_m$-orbit; whence $k[X]^{{\bf G}_m}=k$. In turn, this and \eqref{Ker} yield \eqref{GmGmGm}. %%Note that, in fact, $X$ is affine \cite{Po5}.
\end{proof}
\begin{remark} If $X$ in Corollary \ref{conenen} is normal, then by \cite{Po5}  it is affine.
\end{remark}

\begin{corollary} Let $X$ be a closed subset of $\,{\bf P}^n$ and let
$\,{\widehat X}\subseteq k^{n+1}$ be the affine cone over $X$.
Then ${\rm Ker}^{}_{{\bf G}_m}({\widehat X})=k$.
\end{corollary}

\begin{example}\label{23}  Consider the case $F={\bf G}_m$. If $G$ is a connected reductive subgroup of
${\rm Aut}(X)$ and $X$ contains a dense $G$-orbit, then
Corollary \ref{FGFG} and Example \ref{Gm}(2) imply that \eqref{GmGmGm} holds.
In particular, this is so
for every toric variety $X$; for instance,
\begin{equation*}
{\rm Ker}^{}_{{\bf G}_m}({\bf A}\!^n\times {\bf A}_*^m)=k.
\end{equation*}
(compare with \eqref{=}). Applying this to the varieties considered in
Examples \ref{adj}--\ref{threef}, we see that \eqref{GmGmGm} holds
for every $X$ from the following list:
\begin{gather*}
{\rm Lie}(G)(\alpha_1,\ldots,\alpha_n)\hskip 2mm \mbox{
(see Example \ref{adj})};\\
D_{n,m,r}\hskip 2mm \mbox{(see Example \ref{determ})};\\
X(\lambda_1,\ldots,\lambda_s)\hskip 2mm \mbox{(see Example \ref{S})};\\
({\bf P}^1\times {\bf P}^1)\setminus\Delta,\;\;  {\bf P}^2\setminus C,\;\;  {\bf V}(n_1,\ldots,n_r)\;\; \mbox{where $n_1,\ldots, n_r\geqslant 2$} \hskip 2mm \mbox{(see Example \ref{surf})};\\
{\bf S}_3,\;\;  {\bf S}_4,\;\;  {\bf S}_5,\;\;
{\bf W}_n,\;\;  {\bf P}(A),\;\;  {\bf Q}(B),\;\;
{\bf X}_n^*\;\; \mbox{where $n\neq 0$},\;\;  {\bf Y}^*_1 \hskip 2mm \mbox{(see Example \ref{threef})}.
\end{gather*}

The threefold ${\bf X}_n$ from Example \ref{threef} is homogeneous with respect to ${\bf SL}_{2,|n|}$.
One can prove that
${\bf SL}_{2,|n|}$ is
${\bf G}_m$-generated; whence ${\rm Ker}({\bf X}_n)=k$.

The remaining threefolds ${\bf X}_0^*$, ${\bf Y}_0$, and ${\bf Y}^*_0$ from Example \ref{threef} are considered in Example \ref{remai} below.
\end{example}

The same proof as that of Lemma \ref{otimes} yields
\begin{lemma}\label{ootimes} For any varieties $X_1$
and $X_2$,
\begin{equation*}
{\rm Ker}^{}_F(X_1\times X_2)\subseteq {\rm Ker}^{}_F(X_1)
\otimes_k {\rm Ker}^{}_F(X_2).
\end{equation*}
\end{lemma}

\begin{corollary} \label{cc1}
For any varieties $X_1$ and $X_2$, the following are equivalent:
\begin{enumerate}
\item[\rm(i)] ${\rm Ker}^{}_F(X_1)$ and $\,{\rm Ker}^{}_F(X_2)$ lie in ${\rm Ker}^{}_F(X_1\times X_2)$;
\item[\rm(ii)]
${\rm Ker}^{}_F(X_1\times X_2)= {\rm Ker}^{}_F(X_1)\otimes_k {\rm Ker}^{}_F(X_2)$.
\end{enumerate}
\end{corollary}

\begin{example}\label{remai} Since ${\bf X}^*_0=(({\bf P}^1\times {\bf P}^1)\setminus \Delta)\times {\bf A}\!^1_*$, ${\bf Y}_0=({\bf P}^2\setminus C)\times {\bf A}\!^1$, and
${\bf Y}_0^*=({\bf P}^2\setminus C)\times {\bf A}\!^1_*$ (see Example \ref{threef}), we deduce from Lemma \ref{ootimes} and Example \ref{23} that ${\rm Ker}^{}_{{\bf G}_m}({\bf X}^*_0)={\rm Ker}^{}_{{\bf G}_m}({\bf Y}_0)={\rm Ker}^{}_{{\bf G}_m}({\bf Y}_0^*)=k$.
\end{example}

\begin{lemma}\label{conne}
For any connected algebraic group $F$ that
has no nontrivial characters,
\begin{equation}\label{starr}
k[X]^*\subseteq {\rm Ker}^{}_F(X).
\end{equation}
\end{lemma}
\begin{proof} Let $H$ be the image of an algebraic homomorphism
$F\to {\rm Aut}(X)$.
We claim that
$k[X]^*\subseteq k[X]^H$; by virtue of
Definition \ref{Kerrrrr} this inclusion implies \eqref{starr}.
Since $H$ is connected, every irreducible component of $X$ is $H$-stable, so
proving the claim we may (and shall) assume that $X$ is irreducible. In this case every element of $k[X]^*$ is $H$-semiinvariant by \cite[Theorem 3.1]{PV2}, hence lies in $k[X]^H$ since $H$ has no nontrivial characters. This completes the proof.
\end{proof}

\begin{corollary} Let $\,F$ be a connected algebraic group that
has no nontrivial characters.
Let $X_1$ and $X_2$ be varieties such that
${\rm Ker}_F^{}(X_1)$ and ${\rm Ker}_F^{}(X_2)$ are generated
by units. Then ${\rm Ker}^{}_F(X_1\times X_2)=
{\rm Ker}^{}_F(X_1)\otimes_k {\rm Ker}^{}_F(X_2)$.
\end{corollary}

\begin{lemma}\label{ss}
Let $\,G$ be a connected reductive algebraic group
of rank $\geqslant 2$. Then
every algebraic $G$-module $V$ is a $k$-linear span of the set
\begin{equation}\label{unipo}
 \bigcup_{H\subseteq G} V^H,
\end{equation}
where $H$ in \eqref{unipo} runs over all one-dimensional tori  of $\,G$.
\end{lemma}
\begin{proof} Like in the proof of Lemma \ref{unip} we may (and shall) assume that $V$ is a nonzero simple $G$-module. Let $T$ be a maximal torus of $G$ and let $v\in V$, $v\neq 0$, be a weight vector of $T$. Since $\dim T\geqslant 2$, the $T$-stabilizer $T_v$ of $v$ is a diagonalizable group of dimension $\geqslant 1$. Hence $T_v$ contains a one-dimensional torus $S$. Thus, $v\in V^S$. Like in the proof of Lemma \ref{unip} we then conclude that
the orbit $G\cdot v$ is contained in set \eqref{unipo}. Since
$V$ is simple, the $k$-linear span of $G\cdot v$ coincides with
$V$; whence the claim.
\end{proof}

\begin{theorem}\label{EEEE} Let $X$ be a variety such that ${\rm Aut}(X)$ contains a connected reductive algebraic group $G$ of rank $\geqslant 2$. Then
$${\rm Env}^{}_{{\bf G}_m}(X)=k[X].
$$
\end{theorem}
\begin{proof}
%%Let $G$ be a connected reductive algebraic subgroup in ${\rm %%Aut}(X)$ of rank $\geqslant 2$.
Since the $G$-module $k[X]$ is algebraic, the claim follows from
Lemma \ref{ss} and Definition \ref{Envvv}.
\end{proof}

\begin{remark}{\rm  Clearly, ${\rm Env}^{}_{{\bf G}_m}({\bf A}\!^1)=k$. This shows that in Lemma \ref{ss} and Theorem \ref{EEEE} the condition ``$\geqslant 2$'' can not be repla\-ced by
``$\geqslant 1$''.
}
\end{remark}

\section{Finite automorphism groups  of algebraic varieties}

\subsection{Jordan groups}
\

 The following definition
is inspired by the classical Jordan theorem (Theorem \ref{Jt} below).

\begin{definition} \label{Jd} %%Let $X$ be a algebraic variety.
A group $G$ is called a {\it Jordan group}
if there exists a positive integer
$J_G$, depending only on $G$, such that every finite subgroup
$K$ of $G$ contains a normal abelian subgroup whose index in $K$ is at most $J_G$.
\end{definition}

\begin{remark} Dropping the assumption of normality
in Definition \ref{Jd}
we do not obtain a more general notion.
Indeed, as is known (see, e.g.,\,\cite[Exer.\,12 to Chap.\,I]{Lang}), if a group $P$ contains a
%%finite index
subgroup $Q$
of finite index,
%%such that $[P:Q]\leqslant \infty$,
then there is
%%$Q$ contains
a normal subgroup $N$ of $P$
%%whose index in $P$ is at most $n!$
such that $[P:N]\leqslant [P:Q]!$ and $N\subseteq Q$.
\end{remark}

Jordan's theorem (see \cite[Theorem 36.13]{CR}) can
 be then re\-formulated as follows:

\begin{theorem}\label{Jt} Every
$\,{\bf GL}_n(k)$ is Jordan.
\end{theorem}

\begin{remark} {\rm For $G={\bf GL}_n(k)$, the explicit upper bounds
$J_G$ are known, see \cite[\S36]{CR}.}
\end{remark}

Since subgroups of Jordan groups are Jordan and
every linear algebraic group
is isomorphic to a subgroup of
some $\,{\bf GL}_n(k)$
(see \cite[2.3.7]{Sp}), Theorem \ref{Jt} yields the following more general

\begin{theorem}\label{lag}
 Every linear algebraic group is Jordan.
\end{theorem}

\begin{lemma}\label{G/H} Let $H$ be a finite normal subgroup of a group $\,G$. If $\,G$ is Jordan, then $G/H$ is Jordan.
\end{lemma}
\begin{proof} Let $\pi\colon G\to G/H$ be the natural projection and let $K$ be a finite subgroup of $G/H$. Since $H$ is finite,
$\pi^{-1}(K)$ is a finite subgroup of $G$. As $G$ is Jordan, $\pi^{-1}(K)$ contains a normal abelian subgroup $A$ whose index is at most $J^{}_G$. Hence $\pi(A)$ is a normal abelian subgroup of $K$ whose index is at most $J_G$.
\end{proof}

\begin{lemma}\label{HHH} Let $H$ be a normal torsion-free subgroup of a group $\,G$. If $\,G/H$ is Jordan, then $G$ is Jordan with
$J_{G}=J_{G\!/\!H}$.
\end{lemma}
\begin{proof}  Let $\pi\colon G\to G/H$ be the natural projection and let $K$ be a finite subgroup of $G$. Since $H$ is torsion free, $K\cap H=\{1\}$. Therefore, $\pi|_K\colon K\to \pi(K)$ is an isomorphism. As $G/H$ is Jordan, this implies that $K$ contains a normal abelian subgroup whose index in $K$ is at most $J_{G\!/\!H}$.
\end{proof}

\begin{lemma}\label{product} If the groups $\,G_1$ and $\,G_2$
are Jordan, then $G_1\times G_2$ is Jordan.
\end{lemma}
\begin{proof}
Let $\pi^{}_i\colon G:=G_1\times G_2\to G^{}_i$ be the natural projection and let $K$ be a finite subgroup of $G$. Since $G^{}_i$ is Jordan, $K^{}_i:=\pi^{}_i(K)$ contains an abelian normal subgroup $A^{}_i$ such that
\begin{equation}\label{ne}
[K^{}_i:A^{}_i]\leqslant J^{}_{K_i}.
\end{equation}
The subgroup ${\widetilde A}^{}_i:=\pi^{-1}_i(A^{}_i)\cap K$ is  normal in $K$ and $K/{\widetilde A}^{}_i$ is isomorphic
to $K^{}_i/A^{}_i$. From \eqref{ne} we then conclude that
\begin{equation}\label{nee}
[K:{\widetilde A}^{}_i]\leqslant J^{}_{K_i}.
\end{equation}

Since $A:={\widetilde A}^{}_1 \cap {\widetilde A}^{}_2$ is the kernel of the diagonal homomorphism
$$
K\longrightarrow \prod_{i=1}^2 K/{\widetilde A}^{}_i
$$
determined by the canonical projections $K\to K/{\widetilde A}^{}_i$, we infer from \eqref{nee} that
\begin{equation}\label{neee}
[K:A]=|K/A|\leqslant |\prod_{i=1}^2 K/{\widetilde A}^{}_i|\leqslant J^{}_{K_1}J^{}_{K_2}
\end{equation}
By construction, $A\subseteq A_1\times A_2$. Since $A_i$ is abelian, this implies that $A$ is abelian. As $A$ is normal in $K$, this and \eqref{neee} complete the proof.
\end{proof}

 The following definition distinguishes
  a special class of Jordan groups.

\begin{definition} A group $G$ is called {\it bounded} if there is a positive integer $b^{}_G$, depending only on $G$, such that the order of every finite subgroup of $G$ is at most~$b^{}_G$.
\end{definition}

\begin{examples} \label{bbbb}
\
(1) Finite groups and torsion free groups are bounded.

(2) Every finite subgroup of ${\bf GL}_n({\bf Q})$ is conjugate to a subgroup of ${\bf GL}_n(\bf Z)$ (see, e.g.,
\cite[Theorem 73.5]{CR}). On the other hand,
by Minkowski's theorem (see,\;e.g.,\;\cite[Theorem 39.4]{Hu})
${\bf GL}_n({\bf Z})$
is bounded. Hence
${\bf GL}_n({\bf Q})$ is bounded.
Note that
H. Minkowski and I. Schur obtained explicit upper bounds
of the orders of finite subgroups in ${\bf GL}_n(\bf Z)$, see \cite[\S39]{Hu}.

(3) It is immediate from the definition that every extension of a bounded group by bounded is bounded as well.
\end{examples}

\begin{lemma}\label{JJJJ} Let $\,H$ be a normal subgroup of a
group \,$G$ such that $\,G/H$ is bounded. Then $\,G$ is Jordan if and only if $\,H$ is Jordan.
\end{lemma}

\begin{proof}
A proof is needed only for the sufficiency.
So assume that $H$ is Jordan;
we have to prove that $G$ is Jordan.
Let $K$ be a finite subgroup of $G$. By Definition \ref{Jd}
\begin{equation}\label{KH}
L:=K\cap H
\end{equation}
contains an abelian normal subgroup $A$
such that
\begin{equation}\label{L}
[L:A]\leqslant J_H.
\end{equation}
Let $g$ be an element of $K$. Since $L$ is a normal subgroup of $K$, we infer that
 $gAg^{-1}$
 is a normal abelian subgroup of $L$ and
 \begin{equation}\label{LL}
[L:A]=[L:gAg^{-1}].
\end{equation}

Consider now the group
\begin{equation}\label{M}
M:=  \bigcap_{g\in K} gAg^{-1}.
\end{equation}
It is a normal abelian subgroup of $K$.
We claim that $[K:M]$ is upper bounded by a constant not depending on $K$. To prove this, fix  the representatives
$g_1,\ldots, g^{\ }_{|K/L|}$ of all cosets of $L$ in $K$. Then \eqref{M} and normality of $A$ in $L$ imply that
\begin{equation}\label{MM}
M=\bigcap_{i=1}^{|K/L|} g_iAg_i^{-1}.
\end{equation}
 From \eqref{MM} we deduce that  $M$ is the kernel of the diagonal homomorphism
 \begin{equation*}
 L\longrightarrow \prod_{i=1}^{|K/L|} L/g_iAg_i^{-1}
  \end{equation*}
  determined by the canonical projections $L\to L/g_iAg_i^{-1}$. This, \eqref{LL}, and \eqref{L} yield
  \begin{equation}\label{KLJ}
  [L:M]\leqslant [L:A]^{|K/L|}\leqslant J_{H}^{|K/L|}.
  \end{equation}

  Let $\pi\colon G\to G/H$ be the canonical projection.
  By  \eqref{KH} the finite subgroup
  $\pi(K)$ of
$G/H$ is isomorphic to $K/L$. Since $G/H$ is bounded, this yields
$|K/L|\leqslant b^{}_{G/H}$.
We then deduce from
 \eqref{KLJ} and $[K:M]=[K:L][L:M]$ that
 \begin{equation*}
 [K:M]\leqslant  b^{}_{G/H} J_{H}^{\,b^{}_{G/H}}.
 \end{equation*}
 This completes the proof.
\end{proof}

\begin{remark} There are non-Jordan extensions
(even semidirect products) of
Jordan groups by Jordan ones,
see in Subsection \ref{g} below the discussion of
${\rm Bir}({\bf P}^1\times B)$ where $B$ is an
elliptic curve.
%%: by \cite{Z}, ${\rm Bir}({\bf P}^1\times B)$ where $B$ is an %%elliptic curve, is such an example, see Subsection \ref{g} below.\;
Therefore,  ``bounded'' in Lemma \ref{JJJJ}  cannot be replaced by ``Jordan''.
\end{remark}

\begin{corollary} Let $\,H$ be a finite normal subgroup of a group $\,G$ such that the center of $\,H$ is trivial. If $\,G/H$ is Jordan, then $\,G$ is Jordan.
\end{corollary}
\begin{proof} The conjugating action of $G$ on $H$ determines a homomorphism $\varphi\colon G\to {\rm Aut}(H)$. The definition of $\varphi$ and triviality of the center of $H$ implies that
\begin{equation}\label{normal}
H\cap {\rm ker}\,\varphi=\{1\}.
\end{equation}
 In turn, \eqref{normal} yields that the restriction  of the natural projection $G\to G/H$ to ${\rm ker}\,\varphi$
is an embedding ${\rm ker}\,\varphi \hookrightarrow G/H$. Hence
${\rm ker}\,\varphi$ is Jordan since $G/H$ is Jordan. But $G/{\rm ker}\,\varphi$ is finite since it is isomorphic to a subgroup of ${\rm Aut}(H)$ for the finite group $H$. Whence $G$ is Jordan by Lemma \ref{JJJJ}. This completes the proof.
\end{proof}

We shall now discuss the notion of Jordan group in the frame of
%%the special class of groups, namely, the
automorphism groups of algebraic varieties.

Example \ref{aCre}, Theorems \ref{jtoral}, \ref{equivvv}
and their corollaries below give, for some varieties $X$, the affirmative answer to the following

\begin{question} \label{Jord} Let $X$ be an irreducible
affine
variety. Is it true that ${\rm Aut}(X)$ is Jordan{\rm?}
\end{question}

\begin{example}\label{aCre}  ${\rm Aut}({\bf A}\!^n)$  is Jordan for $n\leqslant 2$. For $n=1$ this is clear, for $n=2$ follows from Theorem \ref{Jt} and the well-known fact that every finite subgroup of ${\rm Aut}({\bf A}\!^2)$ is linearizable, i.e., conjugate to a subgroup of ${\bf GL}_2(k)$ (see also Subsection \ref{g} below).
\end{example}

\begin{theorem} \label{jtoral} The automorphism group of every irreducible
toral
variety $($see Definition {\rm \ref{def_toral}}$)$ is Jordan.
\end{theorem}

\begin{proof} By \cite{R2}, for any irreducible variety $X$, the abelian group
 $$\Gamma:=k[X]^*/k^*$$ is
 free and of
finite rank.
Let $X$ be toral and let
$H$ be the kernel of
the natural action of ${\rm Aut}(X)$ on $\Gamma$.
We claim that $H$ is abelian.
Indeed, for every element $f\in k[X]^*$, the line spanned by $f$ in $k[X]$
is $H$-stable. Since ${\bf GL}_1$ is abelian,
this yields that
\begin{equation}\label{hh}
h_1h_2\cdot f=h_2h_1\cdot f\qquad \mbox{for any elements\quad $h_1, h_2\in H$}.
 \end{equation}
   As $X$ is toral, $k[X]^*$ generates the $k$-algebra $k[X]$ by Lemma
\ref{toral}. Hence \eqref{hh} holds for every $f\in k[X]$. Since $X$ is affine, the automorphisms of $X$ coincide if and only if they induce the same automorphisms of  $k[X]$.  Whence $H$ is abelian, as claimed.

Let $n$ be the rank of $\Gamma$. Then ${\rm Aut}(\Gamma)$ is isomorphic to ${\bf GL}_n({\bf Z})$. By the definition of $H$, the natural action of
${\rm Aut}(X)$ on $\Gamma$ induces an embedding of ${\rm Aut}(X)/H$ into ${\rm Aut}(\Gamma)$.~Hence ${\rm Aut}(X)/H$ is isomorphic to a subgroup of ${\bf GL}_n({\bf Z})$.
Example \ref{bbbb}(2) then implies that ${\rm Aut}(X)/H$ is bounded. Thus,
${\rm Aut}(X)$ is an extension of a bounded group by an abelian group, hence Jordan by Lemma \ref{JJJJ}.
This completes the proof.
\end{proof}

\begin{remark} {\rm
Maintain the notation of the proof of Theorem \ref{jtoral}.
Let $f_1,\ldots f_n$ be a basis of $\Gamma$. There are the homomorphisms $\lambda_i\colon H\to k^*$,  $i=1,\ldots, n$, such that $g\cdot f_i=\lambda(g)f_i$ for every $g\in H$ and $i$.
Since $k[X]^*$ generates $k[X]$, the diagonal map
$
H\to (k^*)^n,\hskip 1mm h\mapsto (\lambda_1(g),\ldots,\lambda_n(g)),
$
is injective. This and the proof of Theorem \ref{jtoral} show that
the automorphism group of $X$ is an extension of a subgroup of
${\bf GL}_n({\bf Z})$ by a subgroup of the torus $(k^*)^n$.
}
\end{remark}

The following lemma is well-known (see, e.g., \cite[Lemma 2.7(b)]{FZ}).

\begin{lemma} \label {linea}
Let $\,X$ be a variety and let $\,G$ be a reductive algebraic subgroup of ${\rm Aut}(X)$. Let $\,x\in X$ be a fixed point of $\,G$. Then the kernel of the induced action of $\,G$ on ${\rm T}_{x, X}$ is trivial.
\end{lemma}

\begin{theorem}\label{equivvv}
Let
$\sim$ be
the equiva\-lence relation
on
the
set of
points of
a variety $\,X$
defined~by
\
\vskip -7mm

\

$$ x\sim y \iff \mbox{the local rings of $\,X$ at $x$ and $y$ are $k$-isomorphic}.
$$

\vskip 1mm

\noindent If there is a finite equivalence class of $\sim$, then
${\rm Aut}(X)$ is Jordan.
\end{theorem}
\begin{proof} Every equivalence class of $\sim$ is ${\rm Aut}(X)$-stable. Let $C$ be a finite equivalence class of $\sim$ and let $G$ be the kernel of the action of ${\rm Aut}(X)$ on $C$. Then $G$ is a normal subgroup of finite index in ${\rm Aut}(X)$. By Lemma \ref{JJJJ} it suffices to prove that $G$ is Jordan.

Let $K$ be a finite subgroup of $G$ and let $x$ be a point of $C$.
As $x$ is fixed by $K$, the action of $K$ on $X$ induces an action of $K$ on ${\rm T}_{x, X}$. The latter is linear and hence determined by a homomorphism $\tau\colon K\to {\bf GL}({\rm T}_{x, X})$. Being finite, $K$ is reductive.  Hence $\tau$ is injective by Lemma \ref{linea}. Theorem \ref{Jt} then yields that $K$ contains an abelian
normal subgroup $A$ such that
$
[K:A]\leqslant J^{\ }_{{\bf GL}_n(k)},\hskip 1mm
n:=\dim {\rm T}_{x, X}.
$
This completes the proof.
\end{proof}

Given a variety $X$, we say that its point $x$ is a {\it vertex} of $X$ if
\begin{equation*}
\dim {\rm T}_{x, X}\geqslant \dim {\rm T}_{y, X}\hskip 2mm\mbox{for every point}\hskip 2mm y\in X.
\end{equation*}
Clearly, an irreducible $X$ is smooth if and only if
every its point is a vertex.

\begin{corollary} \label{nonsmooth}
The automorphism group of every variety
with only finitely many vertices is Jordan.
\end{corollary}

\begin{corollary}\label{coooone}
Let
$\approx$ be
the equiva\-lence relation
on
the
set of
points of
a variety $\,X$
defined~by
\

\vskip -2.8mm

\
$$ x\approx y \iff \mbox{the tangent cones of $\,X$ at $\,x$ and $y$ are isomorphic}.
$$

\vskip 1mm

\noindent If there is a finite equivalence class of $\approx$, then
${\rm Aut}(X)$ is Jordan.
\end{corollary}

\begin{corollary}\label{sssing}  The automorphism group of every nonsmooth
variety with only finitely many  singular points is Jordan.
\end{corollary}

\begin{corollary}\label{connn}
Let
${\widehat X}\subset k^{n+1}$ be the affine cone of a smooth closed proper
subvariety $\,X$ in $\,{\bf P}^n={\bf P}(k^{n+1})$
that does not lie in any hyperplane.
Then
${\rm Aut}(\widehat X)$ is Jordan.
\end{corollary}

\begin{proof} The assumptions imply that the singular locus of $\widehat X$ consists of a single point, the origin; whence the claim by Corollary \ref{sssing}.
\end{proof}

\begin{remark}{\rm  Smoothness in Corollary \ref{connn}  may be replaced by the assumption that $X$ is not a cone. Indeed, in this case the origin constitutes a single  equivalence class of $\approx$ for points of $\widehat X$; whence the claim by Corollary \ref{coooone}.}
\end{remark}

\begin{theorem}\label{aaa} For every variety $X$, every
finite subgroup $\,G$ of ${\rm Aut}(X)$ such that $X^G\neq \varnothing$ contains an abelian normal subgroup
whose index in $\,G$ is at most $J_{{\bf GL}_d(k)}$ where
$d=\underset{x}{\max} \dim {\rm T}_{x, X}$.
\end{theorem}

\begin{proof} Like in the above proof of Theorem \ref{equivvv},
  this follows from Lemma \ref{linea} and Theorem \ref{Jt}.
\end{proof}

 %%The argument also shows that one can take $t^{}_X=J_{{\bf %%GL}_d(k)}$  for $d=\underset{x}{\max} \dim {\rm T}_{x, X}$.

%%\begin{example}
\begin{corollary}\label{ccccccc} Let $p$ be a prime number. Then every finite
$p$-subgroup $\,G$ of
%%the affine Cremona group
${\rm Aut}({\bf A}\!^n)$ contains an abelian normal subgroup
whose index in $\,G$ is at most $J_{{\bf GL}_n(k)}$.
%%$A$ such that $[G:A]\leqslant J_{{\bf GL}_n(k)}$.
%%t^{}_{{\bf A}\!^n}$.
\end{corollary}%%{example}

\begin{proof}
 This follows from Theorem \ref{aaa} since in this case $({\bf A}\!^n)^G\neq \varnothing$, see \cite[Theorem 1.2]{Se3}.
\end{proof}

\begin{remark}{\rm To date, it is not known whether or not
$({\bf A}\!^{n})^G\neq\varnothing$ for every
finite subgroup $\,G$ of ${\rm Aut}({\bf A}\!^{n})$. By Theorem
\ref{aaa} the affirmative answer would imply that ${\rm Aut}({\bf A}\!^{n})$ is Jordan.
}
\end{remark}

\begin{remark} {\rm The statement
of Corollary \ref{ccccccc}
remains true if ${\bf A}\!^n$ is replaced by any $p$-acyclic variety $X$
and $n$ in $J_{{\bf GL}_n(k)}$ by   $\underset{x}{\max} \dim {\rm T}_{x, X}$.
This is because
in this case
 $X^G\neq \varnothing$ for every finite $p$-subgroup $G$ of ${\rm Aut}(X)$,
see \cite[Sect.\;7--8]{Se3}.}
\end{remark}

\begin{theorem}
For every variety $X$,
there is an integer $m^{}_X$
such that any finite subgroup
$\,G$
of  any connected linear
algebraic subgroup $\,L$
of $\,{\rm Aut}(X)$ contains
an abelian normal subgroup whose index in $\,G$ is at most $m^{}_X$.
\end{theorem}

\begin{proof}
 Being reductive, $G$ is contained in a maximal reductive subgroup $R$ of $L$. Then $R$ is a Levi subgroup, i.e., $L$ is a semidirect product of $R$ and  ${\rm Rad}_uL$, cf., e.g., \cite[Chap.\,6]{OV}.
As $L$ is connected, $R$ is connected as well. Since the kernel of the action of
$R$ on $X$ is trivial,
${\rm rk}\,R\leqslant \dim X$, see\;\cite[\S3]{Po2}.
 The claim then follows from Theorem \ref{lag} as
 there are only finitely many connected reductive groups
of rank at most $\dim X$.
\end{proof}

 \subsection{Generalizations}\label{g}
 One may ask whether  ``affine'' in Question~\ref{Jord} can be drop\-ped:

\begin{question} \label{QA}
Is there an irreducible variety $X$
such that ${\rm Aut}(X)$ is not Jordan{\rm?}
\end{question}

The
negative  answer to Question \ref{QA}
would follow from that to

\begin{question} \label{QB}
Is there an irreducible variety $X$
such that ${\rm Bir}(X)$ is not Jordan{\rm ?}
\end{question}

In Theorem  \ref{CS} below we answer Question \ref{QB}
%%Question \ref{QB}
for curves and surfaces.

\vskip 1.5mm

{\it Curves.}

\smallskip

If $X$ is a curve, then the answer to Question \ref{QB}
%%and hence to Question \ref{QA}
is negative.

Proving this
%%Indeed,
we may assume that
$X$ is smooth and projective.\;Then ${\rm Bir}(X)\!=\!{\rm Aut}(X)$.

If $g(X)$, the genus of $X$, is $0$,
then
$X={\bf P}^1$,
hence ${\rm Bir}(X)={\bf PGL}_2(k)$, so
${\rm Bir}(X)$ is Jordan by Theorem \ref{lag}.

If $g(X)=1$, then $X$ is an elliptic curve; whence
${\rm Bir}(X)$ is
the extension of a finite group
by the abelian algebraic group $X$, hence Jordan by Lemma \ref{JJJJ}.

 If $g(X)\geqslant 2$, then ${\rm Bir}(X)$ is finite, hence Jordan.

 Note that all curves (not necessarily smooth and projective)
with infinite automorphism group are classified in~\cite{Po3}.

\vskip 1.5mm

{\it Surfaces.}

\smallskip

Answering Question \ref{QB} for surfaces $X$, we may assume that $X$ is a smooth
projective minimal model.

If $X$ is of general type, then  by
Matsumura's
theorem
${\rm Bir}(X)$ is finite, hence Jordan.

If $X$ is rational, then ${\rm Bir}(X)$ is
the planar Cremona group over $k$, hence Jordan by \cite[Theorem 5.3]{Se1},
\cite[Th\'eor\`eme 3.1]{Se2}.

If $X$ is a nonrational ruled surface, it is birationally isomorphic to ${\bf P}^1\times B$ where $B$ is a smooth projective curve
such that
$g(B)>0$; we may then take $X={\bf P}^1\times B$.
 As $g(B)>0$, there are no dominant rational maps ${\bf P}^1\to B$, hence
the elements of ${\rm Bir}(X)$ permute the fibers of
the natural projection  ${\bf P}^1\times B\to B$.
The set of ele\-ments
 inducing trivial permutation is a normal subgroup ${\rm Bir}^{}_B(X)$ of ${\rm Bir}(X)$.
The definition implies that ${\rm Bir}^{}_B(X)={\bf PGL}_2(k(B))$, hence Jordan by Theorem \ref{lag}.
Naturally identifying ${\rm Aut}(B)$ with the subgroup of ${\rm Bir}(X)$, we get the decomposition
 ${\rm Bir}(X)={\rm Bir}^{}_B(X)\rtimes {\rm Aut}(B)$.
 Note that
${\rm Aut}(X)={\bf PGL}_2(k)\times {\rm Aut}(B)\neq {\rm Bir}(X)$ (see \cite[pp.\;98--99]{Mar}), so ${\rm Aut}(X)$ is Jordan by Lemma \ref{product}.
 Let $g(B)\geqslant 2$. Then ${\rm Aut}(B)$ is finite, hence
$[{\rm Bir}(X):{\rm Bir}^{}_B(X)]<\infty$.
Lemma \ref{JJJJ} then implies that ${\rm Bir}(X)$ is Jordan.
For $g(B)=1$, this argument does not work as $B$ is an elliptic curve
and hence ${\rm Aut}(B)$ is infinite. In fact, by \cite{Z}, if $B$ is an elliptic curve, then
${\rm Bir}(X)$ is {\it not} Jordan (this dispelled a hope expressed in the earlier preprint of the present paper {\tt arXiv:1001.1311v2\,[math.AG]\,6\,Feb\,2010}).

The canonical class of all other surfaces $X$ is numerically effective, so, for them, ${\rm Bir}(X)={\rm Aut}(X)$, cf.\;\cite[Sect.\,7.1, Theorem 1 and Sect.\;7.3, Theorem\;2]{IS}.

Let $X$ be such a surface.\;The group ${\rm Aut}(X)$ has the structure of a locally algebraic group with finite or countably many components, i.e.,  there is a normal subgroup ${\rm Aut}(X)^0$ in ${\rm Aut}(X)$ such that
\begin{enumerate}
\item[(i)]  ${\rm Aut}(X)^0$ is a connected algebraic group; and
\item[(ii)]
${\rm Aut}(X)/{\rm Aut}(X)^0$ is either finite or countable group,
\end{enumerate}
see \cite{Mat}.
By (i) and the structure theorem on algebraic groups \cite{Ba}, \cite{R1}
there is a normal connected linear algebraic subgroup $L$ of ${\rm Aut}(X)^0$ such that
${\rm Aut}(X)^0/L$ is an abelian variety. By \cite[Cor.\,1]{Matm} nontriviality of $L$ would imply that $X$ is ruled. As we assumed that $X$ is not ruled, this means that $L$ is trivial, i.e., ${\rm Aut}(X)^0$ is an abelian variety. Hence, ${\rm Aut}(X)^0$ is an abelian and, therefore, a Jordan group.

By (i) the group ${\rm Aut}(X)^0$ is contained in the kernel of the natural action of ${\rm Aut}(X)$ on $H^2(X, {\bf Q})$ (we may assume that $k={\bf C})$. Therefore, this action defines a homomorphism  ${\rm Aut}(X)/{\rm Aut}(X)^0\to {\bf GL}(H^2(X, {\bf Q}))$. The kernel of this homomorphism is finite by
\cite[Prop.\,1]{Do}, and  the image is a bounded by Example \ref{bbbb}(2).
By Example \ref{bbbb}(1),(3) this yields that ${\rm Aut}(X)/{\rm Aut}(X)^0$ is bounded. In turn, as ${\rm Aut}(X)^0$ is Jordan, by Lemma \ref{JJJJ} this implies that ${\rm Aut}(X)$ is Jordan.

This completes the proof of the following

 \begin{theorem}\label{CS} Let $X$ be an irreducible variety of dimension $\leqslant 2$. Then the following properties are equivalent{\rm:}
 \begin{enumerate}
 \item[\rm(a)] the group ${\rm Bir}(X)$ is Jordan;
  \item[\rm(b)]  the variety $X$ is not birationally isomorphic to ${\bf P}^1\times B$, where $B$ is an elliptic curve.
      \end{enumerate}
\end{theorem}

\bibliographystyle{amsalpha}

\end{document}